\documentclass[12pt]{article}

\usepackage{amsmath}
\usepackage{amsfonts}
\usepackage{latexsym}
\usepackage{amssymb}
\usepackage{csquotes} 
\usepackage{amsthm}
\usepackage{fancyhdr}

\usepackage{lineno,xcolor}
\makeatletter

\newtheorem{theorem}{Theorem}[section]
\newtheorem{proposition}[theorem]{Proposition}

\newtheorem{remark}[theorem]{Remark}

\def\a{\mathbf a}

\def\b{\mathbf b}

\def\c{\mathbf c}

\def\x{\mathbf x}

\def\z{\mathbf z}
\def\0{\mathbf 0}

\def\cB{\mathcal B}
\def\cC{\mathcal C}

\def\cE{\mathcal E}
\def\cF{\mathcal F}

\def\cI{\mathcal I}

\def\cN{\mathcal N}
\def\cK{\mathcal K}
\def\cO{\mathcal O}

\def\cV{\mathcal V}
\def\cT{\mathcal T}
\def\cU{\mathcal U}

\def\cS{\mathcal S}

\def\cV{\mathcal V}

\def\bD{{\mathbb D}}

\def\bS{{\mathbb S}}

\def\PG{{\rm PG}}

\def\Aut{{\rm Aut}}

\def\Fq{\mathbb F_{q}}

\def\Fqm{\mathbb F_{q^{m}}}

\def\Fqn{\mathbb F_{q^n}}

\def\Ftf{\mathbb F_{3^5}}
\def\Ftd{\mathbb F_{3^d}}

\def\ker{{\rm ker}}

\newcommand\comment[1]{}

\def\<{\langle}
\def\>{\rangle}

\title{Eggs in finite projective spaces \\ and unitals in translation planes}

\author{Giusy Monzillo\footnote{The research was supported by the Italian National Group for Algebraic and Geometric Structures and their Applications (GNSAGA-INdAM). } \\
\small {\tt monzillo.giusy@gmail.com}\\[0.8ex]
\small Dipartimento di Matematica, Informatica ed Economia\\[-0.8ex]
\small Universit\`a degli Studi della Basilicata\\[-0.8ex]
\small Viale dell'Ateneo Lucano 10 \\[-0.8ex]
\small 85100 Potenza, Italy\\
\and   
Tim Penttila \\ 
\small {\tt penttila86@msn.com}\\
\small School of Mathematical Sciences\\[-0.8ex]
\small The University of Adelaide\\[-0.8ex]
\small Adelaide, South Australia \\[-0.8ex]
\small 5005 Australia
\and
Alessandro Siciliano$^*$ \\
\small{\tt alessandro.siciliano@unibas.it}\\[0.8ex]
\small Dipartimento di Matematica, Informatica ed Economia\\[-0.8ex]
\small Universit\`a degli Studi della Basilicata\\[-0.8ex]
\small Viale dell'Ateneo Lucano 10 \\[-0.8ex]
\small 85100 Potenza, Italy\\
}
\date{}
\begin{document}
\date{}

\maketitle


%
\begin{abstract}
Inspired by the connection between ovoids and unitals arising from the Buekenhout construction in the Andr\'e/Bruck-Bose representation of translation planes  of dimension at most  two over their kernel, and since eggs  of $\PG(4m-1,q)$, $m\geq1$, are a generalization of ovoids, we explore the relation between eggs and unitals in translation planes of higher dimension over their kernel. By investigating such a relationship, we construct a unital in the Dickson semifield plane of order $3^{10}$,  which is represented in $\PG(20,3)$ by a cone whose base is a set of points constructed from the dual of the Penttila-Williams egg in $\PG(19,3)$. This unital is not polar; so, up to the knowledge of the authors, it seems to be a new unital in such a plane.
\end{abstract}
 


%
%

\section{Introduction}\label{sec_1}

Field reduction has become a theme of finite geometry which turned out very fruitful in the last few decades.  Given a construction of an interesting object from a configuration in a vector space of dimension $r$ over a field of order $q^n$, the question is raised as to which objects give rise similar configurations in a vector space of dimension $rn$ over a field of order $q$.

The Buekenhout-Metz construction of unitals in finite translation planes \cite{b,m} (which gives all known unitals in Desarguesian planes) can be recontextualized in this fashion, with  cones projecting  an ovoid as a base in the  Andr\'e/Bruck-Bose representation of such planes. 

It has long been known \cite{bt} that unitals are extremal in size among minimal blocking sets (at the other end than that most studied - large rather than small). The observation of Lunardon \cite{lun} at the turn of the millennium that changing the field gave access to many more subspaces, some of which were blocking sets, transformed the theory of blocking sets in the process giving rise to the idea of linear sets. Thus, the idea of Buekenhout and Metz was taken  by Sz\H onyi et al. \cite{cgmssw} and, later, by Mazzocca and Polverino \cite{mp} to provide further minimal blocking sets, using cones rather than subspaces. 
 
For the construction by Tits of generalized quadrangles from (ovals and) ovoids, the configurations that arise by applying a field reduction  are eggs, and the similar objects are translation generalized quadrangles. 
Thus, changing the field for ovoids and studying eggs gave the possibility of new translation generalized quadrangles, first realized in work of Kantor \cite{kan} from three decades past; the result of field reduction applied to the concept of an ovoid is an egg.

Motivated by the relationship between ovoids and unitals via the Buekenhout-Metz construction, and since eggs are generalization of ovoids, we explore possible relationships between eggs and unitals. Putting all the above ideas together, in this paper we construct a unital in the Dickson semifield plane of order $3^{10}$,  which is represented in $\PG(20,3)$ by a cone whose base is a set of points constructed from the dual of the Penttila-Williams egg in $\PG(19,3)$. This unital does not arise from a polarity; so it is a new unital, up to the knowledge of the authors.

While field reduction is usually thought of in a projective setting, algebraic dimensions are more amenable to an introductory discussion of it, so we will take a vector space approach along all the paper.


\section{Definitions and preliminary results}

A {\em unital} in a finite projective plane $\pi$ of order $n^2$ is a set $\cU$ of $n^3+1$ points such that  every line of $\pi$ meets $\cU$ in 1 or $n+1$ points. Therefore, $\cU$ is equipped with a family of subsets, each of size $n + 1$, such that every pair of distinct points of $\cU$ is contained in exactly one subset of the family; such subsets are usually called {\em blocks}, and $\cU$ turns out to be a 2-$(n^3+1, n+1, 1)$ design.

In a computer search, Brouwer \cite{bro,bro2}  found a large number of  mutually non-isomorphic 2-$(28, 4, 1)$ designs. Only a few of these are embeddable in a projective plane of order 9 as unitals. One of the examples has been generalized by Gr\"uning \cite{gru}, who constructed a unital of order $q$ for any odd prime power in both the Hall plane and dual Hall plane of order $q$. An infinite family of non-Buekenhout unitals in the Hall planes of order $q^2$ have been constructed in \cite{dov}.
Other infinite families of unitals  in various square order planes are known to exist; see  e.g. \cite{alr}, \cite{barlun}, \cite{bar}, \cite{deha},  \cite{rin}, \cite{ros}.  The only known 2-$(n^3+1, n+1, 1)$ design with $n$ not a prime power is the one found in  \cite{baba} and \cite{math} where $n=6$. For more on 2-$(n^3+1, n+1, 1)$ designs embeddable as unitals in projective plane, see \cite{be}.

In the Desarguesian projective plane $\PG(2,q^2)$, a unital can arise from a unitary polarity: the points of the unital are the absolute points, and the blocks are the intersections of the non-absolute lines of the polarity with $\cU$. These unitals are called {\em classical } or {\em Hermitian unitals}. By a result of Seib  \cite{seib}, the absolute points of a unitary polarity in any square order projective plane form the point-set of a unital.  Such  unitals  are called  {\em polar unitals}.  So, classical unitals of $\PG(2,q^2)$ are  examples of  polar unitals, and   Ganley \cite{gan2} showed that  polar unitals exist in any  Dickson  commutative semifield plane of odd order.

 A {\em finite semifield}  is a finite set $\bS$ with two binary operations $+$ and $*
$, such that $(\bS,+)$ is an abelian group and $(\bS\setminus\{0\},*)$  is a loop such that both distributive laws hold.

Let $\pi(\bS)$ be the  point-line geometry  whose 
points are the elements  in $\bS\times\bS$ and   in $\{(m):m\in\bS\cup \{\infty\}\}$, and  the  lines are the sets 
\[
[m,k]=  \{(y,x)\in \bS\times\bS:m*x+y = k	\}\cup\{(m)\},
\]
\[
[z] = \{(y,z):y \in\bS\}\cup\{(\infty)\}
\]
 and 
 \[
 [\infty] =  \{(m):m \in\bS\}\cup \{(\infty)\} .
 \]
 with $m, k,z \in \bS$, and  $\infty$ a symbol not in $\bS$.
 
  It turns out that $\pi(\bS)$ is a translation plane which  is called the {\em semifield plane coordinatized by} $\bS$. We refer to \cite{bjj} and \cite{dem}  for basic  information on semifields and translation planes. 
 
For any semifield $\bS$, the subset $\cN_l=\{a\in\bS:a*(x* y)=(a* x)* y,\forall x,y\in\bS\}$ is called the  {\em left nucleus} of $\bS$. Similarly, the {\em middle nucleus} $\cN_m$ and the {\em right nucleus} $\cN_r$ are defined. The  set  $\cK=\{a\in\cN_l\cap\cN_m\cap\cN_r:a*b=b*a, \forall b\in\bS\}$ is called the {\em center}  of $\bS$. Each of these four structures is a field, and a finite semifield is a left vector space over its left nucleus and a two-sided vector space over its center \cite{dem}. Here,  $\cK$ is isomorphic to the kernel of the translation plane $\pi(\bS)$. 

For any element $b$ of the semifield $\bS$ with center $\cK$, the map $\phi_b:x\in\bS\mapsto xb\in\bS$ is a linear map when $\bS$ is considered over its left nucleus $\cN_l$. It turns out that the set $\cC_{\bS}=\{\phi_b: b\in\bS\}$ is a $\cK$-vector subspace of the vector space of the $\cN_l$-linear maps of $\bS$. Since $\bS$ is finite, we may assume $\cK=\Fq$, $\cN_l=\Fqn$ and $\bS$ is  an $t$-dimensional left vector space over $\Fqn$, for some positive integers $n$ and $t$.   

Under the previous indentification, the set $\cC_{\bS}$ satisfies the following properties: (i) $\cC_{\bS}$ has $q^{nt}$ elements; (ii) $\cC_{\bS}$ contains the zero and the identity maps; (iii) $A-B$ is non-singular for all distinct $A,B\in \cC_{\bS}$. A set of  linear maps of $V(t, q^n)$ satisfying the above properties  is called a \emph{spread set} of $V(t, q^n)$.

A $(t-1)$-spread of the $(r-1)$-dimensional projective space  $\PG(r-1,q)$ over $\Fq$ is a set $\cS$ of $(t-1)$-dimensional projective subspaces such that  every point is contained in exactly one subspace of $\cS$. It is known that a $(t-1)$-spread of $\PG(r-1,q)$ exists if and only if $t$ divides $r$ \cite{dem}.

Let $\cC$ be a spread set of $V(t,q^n)=\Fqn^{\,t}$. In $\PG(2t-1,q^n)$ consider the subspaces 
\[
S_{\tau}=\{((x_1,\ldots,x_t)^\tau,x_1,\ldots,x_t):x_i\in \Fqn\},
\]
for all $\tau\in\cC$. Then, the set $\cS=\{S_{\tau}:\tau\in \cC \} \cup \{S_\infty\}$, with $S_\infty=\{(x_1,\ldots,x_t,0,\ldots,0):x_i\in \Fqn\}$ forms a $(t-1)$-spread of $\PG(2t-1,q^n)$.

Conversely, let $\cS$ be a $(t-1)$-spread of $\PG(2t-1,q^n)$. Then, it is possible to choose homogeneus coordinates in $\PG(2t-1,q^n)$ such that there is a spread set $\cC$ of $V(t, q^n)$  from which $\cS$ is constructed as above. Thanks to the Andr\'e/Bruck-Bose construction, the spread $\cS$ defines a translation plane $\Pi(\cS)$ \cite{andre, bb, bb2}. If the set $\cC$ is closed under the sum, then there is  a (finite) semifield $\bS$ that coordinatizes $\Pi(\cS)$ such that $\cC=\cC_{\bS}$; the  left nucleus of $\bS$ is $\Fqn$ and $\bS$ can be viewed as a $t$-dimensional left vector space over $\Fqn$ \cite{dem}. In addition, if $\mathbb{F}_q$ is the largest subfield $K$ of $\Fqn$ such that $\cC$ is a $K$-vector subspace of the vector space of the $\Fqn$-linear maps of $V(t,q^n)$, the center of $\bS$ is $\mathbb{F}_{q}$. Therefore, there exists a canonical correspondence between translation planes coordinatized over a semifield $\bS$ with dimension $t$ over its left nucleus $\Fqn$ and center $\Fq$, and the $(t-1)$-spreads of $\PG(2t-1,q^n)$ arising from a spread set of $V(t,q^n)$, that is closed under the sum. Moreover, it is well-known that the resulting plane is Desarguesian if and only if $\cS$ is a Desarguesian spread \cite{bb2}.

Buekenhout  \cite{b}, and  Metz \cite{m} (by refining Buekenhout's idea), constructed unitals in any translation planes  with dimension at most two over their kernel by using the Andr\'e/Bruck-Bose representation of such planes. These unitals  are cones of $\PG(4,q)$ projecting  an ovoid in a 3-dimensional subspace of $\PG(4,q)$ from a point at infinity.  
These unitals  are called  {\em Buekenhout-Metz unitals}. Since classical unitals can be obtained in  this way, they fall in the class of  Buekenhout-Metz unitals which, so far, are the only  known unitals of $\PG(2,q^2)$.

Many other authors have used the above representation of $\PG(2, q^n)$ in $\PG(2n, q)$ to study objects in the Desarguesian plane  in order to determine whether
this higher dimensional representation provides additional information about those objects in the plane. In particular, the projective plane $\PG(2,q^4)$, modelled in $\PG(8,q)$, has been considered in \cite{bcq} to study the representation of  classical unitals, and the representation of $\PG(2,q^{2m})$  in $\PG(4m,q)$, for $m>1$, have been considered to study other geometric objects of the plane;  see \cite{mp, mps, rsv, cgmssw} just to cite some.



A {\em blocking set} in a projective plane $\pi$ is a set of points such that every line of $\pi$ has a  non-empty intersection with the set. A blocking set is said to be {\em minimal} if through any of its points there is a line of $\pi$ intersecting it precisely in that point.

In the paper \cite{bt}, Bruen and Thas proved that, when the order of the projective plane is a square, say $n^{2}$, then  the size of a minimal blocking set is bounded by $n^{3}+1$. This size is reached if and only if the minimal blocking set is a unital.  

In \cite{mp} the following geometric setting was introduced to construct large minimal blocking sets of $\PG(2,q^{2m})$ from cones in its Andr\'e/Bruck-Bose representation in  $\PG(4m,q)$.
Let $\z$ be a fixed element of a $(2m-1)$-spread $\cS$ of $\Sigma_{\infty}$ and  $\cV$  an $(m-1)$-dimensional subspace of $\z$. Let  $\Gamma$ be a $(3m-1)$-dimensional subspace of $\Sigma_\infty$ disjoint from $\cV$. For every $\x\in\cS$, $\x\neq \z$,  let $I(\x)$  be the $(2m-1)$-dimensional subspace $\<\x,\cV\>\cap\Gamma$. We denote by  $\cI(\cV)$ the set of all the subspaces $I(\x)$, $\x\in\cS$. 
Let $\Gamma'$ be an  affine $3m$-dimensional subspace of $\PG(4m,q)$ through $\Gamma$, and denote by $\cF(\cV)$ the set of all affine $2m$-dimensional subspaces of $\Gamma'$ containing an element of  $\cI(\cV)$.

Let $\cF$ be a family of $2m$-dimensional subspaces of $\Gamma'$. 
An $\cF$-{\em blocking set} of $\Gamma'$ is a set $\cB$ of affine points such that every element of $\cF$ has a non-empty intersection with $\cB$. The blocking set $\cB$ is said to be {\em minimal} if through any  point of $\cB$ there is an element in $\cF$ intersecting $\cB$ precisely in that point.

By keeping the above geometric setting in mind, the following result, which is a sharpening of Corollary 3.3 in  \cite{mp}, is crucial for our succeeding considerations.

 \begin{proposition}\label{prop_1}
Let  $\cB$ be a set   of affine points of $\Gamma'$  and 
\begin{equation}\label{eq_1}
\cB^*=\bigcup_{ P\in\cB}{\langle \cV,P\rangle}\cup\{\z\}.
\end{equation}
If  $\cB$ is a minimal $\cF(\cV)$-blocking set, then $\cB^*$ is a minimal blocking set of  size $|\cB^*|=q^{m}|\cB|+1$ in the translation plane $\Pi(\cS)$.
\end{proposition}
\begin{proof}
Construction 2 in \cite{mp} works perfectly well under the milder hypothesis that $\cS$ is any $(2m-1)$-spread of $\Sigma_\infty$. The details are left to the reader.
\end{proof}

\comment{
We can reverse the previous result.
\begin{proposition} \label{prop_6}
Let $\cB^*$ be a set of affine $d-$dimensional subspaces in $\PG(4m,q)$ (called "generators" of $\cB^*$)
 all through the $(m-1)$-subspace $\cV$ of $\z$.  If $\cB^*$ is a minimal blocking set in $\Pi(\cS)$ then the set
\[
\cB=\{\pi\cap\Gamma': \pi \mathrm{\ is\ a\ generator\ of\ } \cB^*\}\cup\{\Theta\}.
\]
is a minimal  $\cF(\cV)-$blocking set in $\Gamma'$ of size $(|\cB^*|-(q^{m-1}+\ldots+q+1))/q^d+q^{m-1}+\ldots+q+1$.
\end{proposition}

\begin{proof}
We first note that $\pi\cap\Gamma'$ is an affine point for any generator $\pi$ of $\cB^*$.

Let $I_{2d}\in\cF(\cV)$. If $I_{2d}\subset\Gamma$ then $I_{2d}\cap\Theta$ is a point since each element of $\cI(\cV)$ is disjoint from $\Theta$. Assume $I_{2d}\not\subset\Gamma$ and set $I(\x)=I_{2d}\cap\Gamma$, for some $\x\in\cS\setminus\{\c\}$. Let $L$ be a line through $\x$ of $\Pi(\cS)$. As $\cB^*$ is a blocking set in $\Pi(\cS)$ the there is a point $P$ in $L\cap\cB^*$. We also note that $L$ intersects the generator $\<\cV,P\>$ of $\cB^*$ precisely at $P$ since $L$ is disjoint from $\cV$. The $3m$-dimesional subspace spanned by $\x$, $\cV$ and $P$ intersects $\Gamma'$ in the $2m$-dimensional subspace $I_{2d}$ giving $|I_{2d}\cap\cB|=|L\cap\cB^*|$.

It easy to see that
\[
|\cB|=\frac{|\cB^*|-(q^{m-1}+\ldots+q+1)}{q^d}+|\Theta|.
\]
\end{proof}
}

By combining the above result of Bruen and Thas with Proposition \ref{prop_1}, we get  the following theorem.

\begin{theorem}\label{th_2}
Let  $\cB$  be a minimal $\cF(\cV)$-blocking set of size $q^{2m}$.   Then, the cone $ \cB^*$ defined in Proposition $\ref{prop_1}$ is a unital in $\Pi(\cS)$.
\end{theorem}

If $\cS$  is a Desarguesian $(2m-1)$-spread of $\Sigma_\infty$, then  there is  a unique  Desarguesian $(m-1)$-spread, say $\cT$, that fills every element of $\cS$, i.e., $\cT$ induces a $(m-1)$-spread in each spread element of $\cS$ \cite{drudge}. The following result gives a characterization of Buekenhout-Metz unitals as cones in $\PG(4m,q)$.

\begin{proposition}\cite{mp}\label{prop_3}
Let   $\cS$  be a Desarguesian $(2m-1)$-spread of $\Sigma_\infty$ and  $\cB$  a minimal $\cF(\cV)$-blocking set of size $q^{2m}$. Then, the cone  $\cB^*$ is a Buekenhout-Metz unital in  $\PG(2,q^{2m})$  if and only if $\cV$ is an element of the spread $\cT$.
\end{proposition}

\comment{
It turns out that a minimal $\cF(\cV)-$blocking set $\cB$ of size $q^{2m}$ with $\cV$  not in $\cT$ provides a  non Buekenhout-Metz unital in $\PG(2,q^{2m})$.
}

\section{Unitals from eggs}\label{sec_3}

 An {\em egg}  in $\PG(4m-1,q)$ is a set $\cE$ of $q^{2m}+1$ pairwise disjoint  $(m-1)-$dimensional subspaces such that any three egg elements span a $(3m-1)-$dimensional subspace. When $m=1$, this definition  recovers indeed the notion of  ovoid in $\PG(3,q)$.  Therefore, since the notion of an egg, introduced by J.A. Thas in \cite{thas}, generalizes that of an ovoid, it make sense to investigate whether it is possible to mimic Buekenhout's construction to get unitals in translation planes with dimension over their kernel greater than two, by using eggs.  Apart from the so-called {\em elementary eggs}, which are obtained by applying the field reduction to an ovoid in $\PG(3,q^m)$,  there are few other known examples of eggs, namely, the Kantor-Knuth eggs, the Cohen-Ganley eggs and the (sporadic) Penttila-Williams egg; see \cite{lp} for an explicit description of these objects. 
 

Let $\cE$ be an egg in $\PG(4m-1, q)$. For every egg element $E$ there exists a unique $(3m - 1)$-dimensional subspace, denoted by $E^*$, containing $E$ and disjoint from any other	egg element; it is called the {\em tangent space of} $\cE$ at $E$.
Therefore, the egg $\cE$ defines an egg in the dual space of $\PG(4m-1,q)$, called the {\em dual egg} of $\cE$ and denoted by $\cE^D$. 

The following result is a corollary of Theorem \ref{th_2}.
\begin{theorem}\label{th_3}
Let $\cE$ be an egg in $\PG(4m-1,q)$,  and $E_\infty$ a fixed egg element.   Let $\Gamma'$ be a  $3m$-dimensional subsubspace of $\PG(4m-1,q)$ containing the tangent space $E_\infty^*$ at $E_\infty$.  In $\Gamma'$ we consider the sets:  
 \[
 \cB_{\cE}=\{ E\cap \Gamma': E\in \cE,  E\neq  E_\infty\}
 \]
and 
\[
\cI_{\cE}=\{E^* \cap E_\infty^*: E\in \cE,  E\neq  E_\infty\}.
\]
Let $\cF_{\cE}$ be the family of all affine $2m$-dimensional subspaces of $\Gamma'$ containing an element of $\cI_{\cE}$, and assume that  $\cB_\cE$ is a minimal $\cF_{\cE}$-blocking set.  

Embed $\Gamma'$ in $\PG(4m,q)$ in such a way that $E_\infty^*$ is a subspace of the hyperplane at infinity  $\Sigma_\infty$ of $\PG(4m,q)$, and $\Gamma'$ is an affine subspace. 

If there exist  a $(2m-1)$-spread $\cS$ of $\Sigma_\infty$ and  a $(m-1)$-dimensional subspace $\cV$ disjoint from $E_\infty^*$ and contained in a spread element $\z$ such that $\cI_{\cE}=\cI(\cV)$, then the cone 
\[
\cB^*=\bigcup_{ P\in\cB_{\cE}}{\langle P,\cV\rangle}\cup\{\z\}
\]
is a unital in  $\Pi(\cS)$.  
\end{theorem}
\begin{proof} 
Here, $\cB_{\cE}$ is a set of  $q^{2m}$ points of $\Gamma' \setminus E^*_{\infty}$, and hence it consists of affine points of $\PG(4m,q)$. Furthermore, every element in $\cI_{\cE}$ is a $(2m-1)$-dimensional subspace of $E^*_{\infty}$. By Theorem \ref{th_2}, if $\cF_{\cE}$ coincides with the family $\cF(\cV)$ previously defined, then $\cB^*$ is a unital in the semifield plane $\Pi(\cS)$. Since $\cF_{\cE}$ consists of all affine $2m$-dimensional subspaces of $\Gamma'$ through an element of $\cI_{\cE}$, we get that   $\cF_{\cE}=\cF(\cV)$ if and only if $\cI_{\cE}=\cI(\cV)$. 
\end{proof}

An egg  is said to be {\em good at an element } $E$ if every $(3m-1)$-dimensional subspace containing $E$ and at least two other egg elements, contains exactly $q^m+1$ egg elements \cite{thas2}.

Let $\cK$ be the quadratic cone in $\PG(3,q^m)$	with equation $X_0X_1=X_2^2$. A {\em flock} of $\cK$ is a set of $q^m$ planes partitioning the cone minus  its vertex $V=\<(0,0,0,1)\>$ into disjoint conics. 
In accordance with this choice of coordinates, the planes of a flock of $\cK$ can be written as $tX_0+f(t)X_1+g(t)X_2+X_3=0$, for all $t\in\Fqm$, for some $f,g:\Fqm\rightarrow\Fqm$. We denote this flock by $\cF(f,g)$. If $f$ and $g$ are linear over a subfield of $\Fqm$, then the flock is called a {\em semifield} flock. The maximal subfield with this property is called the {\em kernel} of the flock. 

From now on, we assume that the kernel of a semifield flock $\cF(f,g)$ is $\Fq$. This implies that the $f$ and $g$ are $\Fq$-linearized polynomials, i.e.
\[
f(t)=\sum_{i=0}^{m-1}{c_it^{q^i}},\,\,\,\,\, g(t)=\sum_{i=0}^{m-1}{b_it^{q^i}},
\]
for some $b_i,c_i\in\Fqm$, $i=0,\ldots,m-1$.

If a basis of $\Fqm$ over $\Fq$ is fixed, then every $r$-ple $(x_1,\ldots,x_r)\in \Fqm^r$ can be viewed as a $rm$-ple over $\Fq$, which will be denoted  by $(x_1,\ldots,x_r)_q$.  In the paper \cite{lp} it was shown that for every semifield flock $\cF(f,g)$ there corresponds  an egg in  $\PG(4m-1,q)$ whose dual, say $\cE$, is good at an element, which can be assumed to be $E_{\infty}$.  Then, the elements and the tangent spaces of $\cE$  have the following form, respectively:
\begin{equation}\label{eq_4}
\begin{array}{ll}
 E(a,b)=\{( t,-g_{(a,b)}(t), -at,-bt)_q:t\in\Fqm\}, \mathrm{ \ for \ all\ }a,b\in\Fqm,\\[.1in]
E_\infty=\{(0,t,0,0)_q:t\in\Fqm\},\\[.1in]
E^*(a,b)=\{(t,h_{(a,b)}(r,s)+g_{(a,b)}(t),r,s)_q:t,r,s\in\Fqm\}, \mathrm{\ for \ all\ }a,b\in\Fqm,\\[.1in]
E^*_\infty=\{(0,t,r,s)_q:t,r,s\in\Fqm\},
\end{array}
\end{equation}
with  
\[
g_{(a,b)}(t)= a^2t+\sum_{i=0}^{m-1}{(b_{i} ab+c_{i} b^2)^{1/{q^i}}t^{1/{q^i}}},\]
and 
\[h_{(a,b)}(r,s)=2ar+\sum_{i=0}^{m-1}{(b_{i} (as+br)+2c_{i} bs)^{1/{q^i}}}.
\]
Because of the expression of the polynomials $g(a,b)$ and $h(a,b)$, such an egg will be denoted by $\cE(\b,\c)$.

\comment{An egg $\cE$ in $\PG(4m-1,q)$ together with the set $T_\cE$ of the tangent spaces  is said to be of {\em additive type}  if the elements of $\cE$ have the form 
\[
\begin{array}{ll}
E(a,b)=\{\<(t,-g_{(a,b)}(t),-(a,b)^{\delta_t})\>_q:t\in\Fqm\}, \mathrm{\ for \ all\ }a,b\in\Fqm\\[.1in]
E_\infty=\{\<(0,t,0,0)\>_q:t\in\Fqm\}
\end{array}
\]
and the elements of $T_\cE$ have the form
\[
\begin{array}{ll}
E^*(a,b)=\{\<(t,ar+bs+g_{(a,b)}(t),r,s)\>_q:r,s,t\in\Fqm\}, \mathrm{\ for \ all\ }a,b\in\Fqm\\[.1in]
E_\infty^*=\{\<(0,t,r,s)\>_q:r,s,t\in\Fqm\},
\end{array}
\]
where  
\[
\begin{array}{ll}
g_{(a,b)}(t)=\sum_{i=0}^{m-1}{(a_ia^2+b_iab+c_ib^2)t^{q^i}}\\[.1in]
(a,b)^{\delta_t}=(\sum_{i=0}^{m-1}{(2a_ia+b_ib)t^{q^i}},\sum_{i=0}^{m-1}{(b_ia+2c_ib)t^{q^i}}),
\end{array}
\]
for some $a_i\,b_i,c_i\in\Fqm$.
If an egg can be written in this form then the egg will  be denoted by $\cE(\a,\b,\c)$, where $\a=(a_0,\ldots, a_{m-1})$, $\b=(b_0,\ldots, b_{m-1})$ and $\c=(c_0,\ldots, c_{m-1})$. 

The following result characterizes good eggs in $\PG(4m-1,q)$ when $q$ is odd.

\begin{theorem}\cite[Theorem 3.4]{lp}\label{th_10}
An egg $\cE$ in $\PG(4m-1,q)$, $q$ odd, is good at an element if and only if $\cE$ is the dual of an additive type egg.
\end{theorem}
}
 
\begin{theorem}\label{th_11}
 Let $\cE=\cE(\b,\c)$ be a good egg of $\PG(4m-1,q)$, which is good at $E_\infty$. Then, the set $\cB_{\cE}$  is a minimal $\cF_{\cE}$-blocking set in $\Gamma'=\PG(3m,q)$  if and  only if $ X^2+\sum_{i=0}^{m-1}{(b_{i} XY+c_{i} Y^2)^{1/{q^i}}}+c=0$ has a solution for all $c\in\Fqm$.
 \end{theorem}
\begin{proof}
Let $\Gamma'=\{(u,t,r,s)_q:u\in\Fq \hbox{ and } r,s,t\in \Fqm\}$. It is evident that $\Gamma'$ is a projective space of dimension $3m$ over $\Fq$ and it contains $E^*_\infty$. By  taking into account the general form of the elements of $\cE=\cE(\b,\c)$, we get 
\[
\cB_{\cE} =  \{\<(1,-g_{(a,b)}(1),-a,-b)_q\>:a,b\in \Fqm\}
 \]
and
\[
\cI_{\cE} =   \{I(a,b):a,b\in \Fqm\},
\]
where 
\begin{equation}\label{eq_27}
 I(a,b)= E^*(a,b)\cap E^*_{\infty}=\{(0,h_{(a,b)}(r,s),r,s)_q:r,s\in\Fqm\}.
\end{equation} 

All the affine $2m$-dimensional subspaces of $\Gamma'$ through an $I(a,b)$ are determined by joining it with an affine point of the affine $m$-dimensional subspace spanned by $E_{\infty}$ and $O=\<(1, 0, 0, 0, 0)\>$.
Therefore, the elements of  $\cF_{\cE}$ have the form  
\[
 F(a,b,c)=\{(u,uc+h_{(a,b)}(r,s),r,s)_q:u\in\Fq \hbox{ and } r,s\in\Fqm\},
\]
for all $a,b,c\in\Fqm$.

A point $P(x,y)=\<(1,-g_{(x,y)}(1),-x,-y)_q\>\in \cB_{\cE}$ lies in $F(a,b,c)$ if and only if 
\[
-g_{(x,y)}(1)=-h_{(a,b)}(x,y)+c
\]
or, equivalently, if and only if $(x,y)$ is a solution of 
\begin{equation}\label{eq_8}
X^2+\sum_{i=0}^{m-1}{(b_{i} XY+c_{i} Y^2)^{1/{q^i}}}-2aX-\sum_{i=0}^{m-1}{(b_{i}(aY+bX)+2c_{i}bY)^{1/{q^i}}}+c=0.
\end{equation}

We refer to the polynomial on the left-hand side of the equation as $H_{(a,b,c)}(x,y)$.
Since $\cE$ is an egg, for any given  $a,b\in\Fqm$, the intersection of the tangent space $E^*(a,b)$ with $\Gamma'$ is the $2m$-dimensional subspace  $F(a,b, c')\in\cF_{\cE}$, with $ c'=g_{(a,b)}(1)$. Therefore,  through the point $P(a,b)\in\cB_{\cE}$ there is the element $F(a,b, c')\in\cF_{\cE}$ intersecting $\cB_{\cE}$ precisely  at $P(a,b)$. 
 This implies that $\cB_{\cE}$ is a minimal $\cF_{\cE}$-blocking set  if and only  if  Eq. (\ref{eq_8}) has a solution $(x,y)\in\Fqm\times\Fqm$ for any given  elements  $a,b,c \in \Fqm$. 

From \cite[Lemma 1.4]{ml}, for any $a,b \in \Fqm$, the linear collineation
\[
\begin{array}{cccc}
\psi_{a,b}: & \PG(4m-1,q) & \longrightarrow & \PG(4m-1,q) \\
            &\<(u,t,r,s)_q\> & \mapsto & \<(u,t+h_{(a,b)}(r,s)-g_{(a,b)}(u),r-ua,s-ub)_q\>
\end{array}
\]
fixes $E_\infty$ pointwise and maps $E(a',b')$ to $E(a'+a,b'+b)$. In addition, $\psi_{a,b}$ fixes $\Gamma'$, and hence $\cB_{\cE}$.  A straightforward, though tedious, calculation shows that $\psi_{a,b}$ acts also on the set of tangent spaces by fixing $E^*_\infty$ setwise and mapping  $ E^*(a',b')$ to $ E^*(a+a',b+b')$.  This implies that $\psi_{a,b}$ fixes both $\cI_{\cE}$ and $\cF_{\cE}$ setwise; in particular, $ F(a',b',c)$ is mapped to $ F(a+a',b+b',c')$, with $c'=c-g_{a,b}(1)+ h_{(a'+a,b'+b)}(a,b)$. This means that, because of the linearity of the second sum in Eq. (\ref{eq_8}),  $H_{(a'+a,b'+b,c)}(x,y)=0$ has a solution for all $c\in\Fqm$ if and only if $H_{(a,b,c)}(x,y)=0$ has a solution for all $c\in\Fqm$. Therefore, $\cB_{\cE}$ is a minimal $\cF_{\cE}$-blocking set  if and only if, for a fixed pair  $(a,b)\in\Fqm\times \Fqm$,  Eq. (\ref{eq_8}) has at least one solution $(x,y)\in\Fqm\times \Fqm$, for all $c\in\Fqm$.  In particular, we can chose $(a,b)=(0,0)$ so that Eq. (\ref{eq_8}) reduces to
\begin{equation}\label{eq_8_1}
X^2+\sum_{i=0}^{m-1}{(b_{i} XY+c_{i} Y^2)^{1/{q^i}}}+c=0.
\end{equation}
\end{proof}


\section{A new unital in a Dickson commutative semifield plane}

In \cite{pw}, Penttila and  Williams constructed an ovoid of the  parabolic quadric $Q(4,3^5)$ in $\PG(4,3^5)$, i.e., a set $\cO$ of $3^{10}+1$ points having exactly one point on each generator of the quadric. Moreover, $\cO$ is a translation ovoid, meaning that the points of $\cO$ can be coordinatized by using functions that are additive over $\mathbb{F}_3$. According to a construction given in \cite{lun2}, such a translation ovoid  corresponds to a semifield flock of the quadratic cone in $\PG(3,3^5)$, which, in turn,  corresponds to a generalized quadrangle with parameters $(3^{10},3^5)$, whose point-line dual is a translation generalized quadrangle. By a result of Payne and Thas \cite[8.7.1]{pt}, the latter generalized quadrangle is isomorphic to $T(\cE)$ for some egg $\cE$ in $\PG(19,3)$. By Theorem 3.4 in \cite{lp}, the dual egg of $\cE$ forms a good egg $\cE^D$ in $\PG(19,3)$.
Whence, via the above correspondences, the Penttila-Williams ovoid of $Q(4,3^5)$ gives rise to a good (dual) egg in $\PG(19,3)$. 
In order to simplify the notation, we will refer to it as  $\cE=\cE(\b,\c)$ with $\b=(0,1,0,0,0)$, $\c=(0,0,0,-1,0)$; see \cite{lp}. 

According to the expressions of the polynomials $g_{(a,b)}(t)$ and $h_{(a,b)}(r,s)$ in this case,  the egg elements  of $\cE$ are defined by the polynomials 
\begin{equation}\label{gpw}
g_{(a,b)}(t)=a^2t-(b^2)^{3^2}t^{3^2}+(ab)^{3^4}t^{3^4}
\end{equation}
and
\begin{equation}\label{hpw}
h_{(a,b)}(r,s)=-ar+b^{3^2}s^{3^2}+(br+as)^{3^4},
\end{equation}
for all $a,b\in\Ftf$.

Let $p$ be an odd prime and $\xi$  a non-square in $\mathbb{F}_{p^m}$. By \cite[p.241]{dem}, the multiplication defined by 
\[
(x,y)*(a,b)=(ax+\xi b^{\alpha}y^{\alpha},bx+a y) 
\]
with $\alpha\in\Aut(\mathbb{F}_{p^m})$  not  the identity, turns $\mathbb{F}_{p^m}^2$ into a Dickson commutative semifield of order $p^{2m}$ which we denote by $\bD=\bD(p^m,\xi,\alpha)$. In particular, its middle nucleus is $\cN_m=\{(a,0): a\in\mathbb{F}_{p^m}\}$, and its left nucleus is $\cN_l=\{(a,0): a\in \mathrm{Fix}(\alpha)\}$, coinciding with its center $\cK$.
 
Now, let $p=3$ and $m=5$. 
For any pair $(a,b)\in\Ftf^2$, we consider the following map
\[
\begin{array}{cccc}
\tau_{(a,b)}: & (x,y) & \mapsto & (bx+ay,-ax+b^{3^2}y^{3^2}),
\end{array}
\]
which defines  the subspaces
$S(a,b)=\{((x,y)^{\tau_{(a,b)}},x,y)_3:x,y\in \Ftf\}$
of $\PG(19,3)$. Set  $\cS=\{S(a,b):a,b \in\Ftf\}\cup \{S_\infty\}$, where $S_\infty=\{(x,y,0,0)_3: x,y\in\Ftf\}$.

 Let $\varphi$ be the linear map $\varphi:(x,y)\mapsto(-y,x)$.  Then, the set  $\{\varphi\tau_{(a,b)}:a,b\in\Ftf\}$ is precisely the spread set of $\mathbb{F}_3^{10}$ associated with the Dickson commutative  semifield $\bD=\bD(3^5,-1,3^2)$.

It turns out that $\cS$ is a 9-spread of $\Sigma_\infty=\PG(19,3)$ and, by \cite{alb}, the translation plane $\Pi(\cS)$ is isomorphic to the Dickson commutative semifield plane $\pi(\bD)$.

Let  $\cV=\{(t,-t^{3^4},0,0)_3:t\in\Ftf\}$. Then, $\cV$ is contained in the spread element $\z=S_\infty$ and it intersects trivially the subspace $\Gamma=E^*_\infty=\{(0,t,r,s)_3:t,r,s\in\Ftf\}$. We also have
\[
\<S(a,b),\cV\>=\{(br+as+t,-ar+b^{3^2}s^{3^2}-t^{3^4},r,s)_3:t,r,s\in\Ftf\},
\]
giving
\[
\<S(a,b),\cV\>\cap\Gamma= \{(0,-ar+b^{3^2}s^{3^2}+(br+as)^{3^4},r,s)_3:r,s\in\Ftf\}
\]
which is precisely the subspace $I(a,b)$ defined by expression (\ref{eq_27}), with $h_{(a,b)}(r,s)$ as in (\ref{hpw}).
\begin{proposition}
The set $\cB_{\cE}$ defined by the Penttila-Williams egg $\cE=\cE(\b,\c)$ is a minimal $\cF(\cE)$-blocking set. 
 \end{proposition}
 \begin{proof}
 By taking into account Theorem \ref{th_11}, $\cB_{\cE}$  is a minimal $\cF_{\cE}$-blocking set  if and  only if 
 \begin{equation}\label{eq_5}
  X^2+ (XY)^{3^4}-(Y^2)^{3^2}=-c
  \end{equation} 
has a solution for all $c\in\Fqm$. 

We distinguish two cases:  $-c$ is a square in $\Fqm$ or not. If $-c$ is a square, then $(\pm \sqrt{-c},0)$ are solutions of Eq. (\ref{eq_5}); if $-c$ is not a square, then $(0,\pm \sqrt{c^{3^3}})$ are solutions of Eq. (\ref{eq_5}).
\end{proof} 
By Theorem \ref{th_3}, the cone 
\[
\cB^*_{\cE}=\{\<(1,c, -g_{(a,b)}(1)-c^{3^4},-a, -b)_3\>:a,b,c\in\Ftf\}\cup \{S_\infty\},
\]
with $ g_{(a,b)}(t)$ as in (\ref{gpw}), is  a unital in the translation plane  $\Pi(\cS)$.

Consider the collineation of $\PG(20,3)$ defined as $\overline\varphi:\<(u,v,t,r,s)_3\>\mapsto 	\<(u,-t,v,r,s)_3\>$. Then, $\Pi(\cS)^{\overline\varphi}$ represents the Dickson commutative semifield plane  $\pi(\bD)$. It turns out that the set  
\[
\cU=\{(g_{(a,b)}(1)+c^{3^4},c,-a, -b):a,b,c\in\Ftf\}\cup \{(\infty)\}
\]
 is a unital in $\pi(\bD)$. 
 Note that $\cU$ cannot be a Buekenhout-Metz unital since $\pi(\bD)$ is a 10-dimensional translation plane over its kernel $\mathbb{F}_{3}$. On the other hand, as $\pi(\bD)$ admits unitary polarities \cite{gan2},  $\cU$ might be a polar unital.
The following result shows that this is not the case.

\begin{theorem}
The unital $\cU$  is not a polar unital in $\pi(\bD)$.
\end{theorem}
\begin{proof}
Since the tangent space at the egg element $E(0,0)$ is $E^*(0,0)$, the tangent line of $\Pi(\cS)$ at the point $O=\<(1,0,0,0,0)_3\>\in\cB^*_{\cE}$ is the subspace spanned by $S(0,0)$ and $O$.  Then,  the tangent line of $\pi(\bD)$ at $(0,0)\in\cU$ is $[0,0]$.

 From \cite[Theorem 2.1]{hltw}, any unitary polarity of $\pi(\bD)$ mapping $(0,0)$ to $[0,0]$ is given by
\[
\begin{array}{cccc}
\rho_a: &(x_1,x_2,y_1,y_2) & \leftrightarrow &[a x_1,-a x_2,-y_1,y_2], \\
              & (m_1,m_2)& \leftrightarrow & (a^{-1}m_1,-a^{-1}m_2)\\
              & (\infty) &  \leftrightarrow & [\infty].
\end{array}
\]
for some non-zero $a\in\Ftf$.

The unital $\cU$ is a polar unital with respect to $\rho_a$,   for some $a\in\Ftf$,  if and only if each  of its points is an absolute point. Straightforward calculations show that the point  $(1,1,0,0)\in\cU$ is not incident with $\rho_a(1,1,0,0)=[a,-a ,0,0]$ for all non-zero $a\in\Ftf$, showing that $\cU$ is not a polar unital.
 \end{proof}


\begin{thebibliography}{10}

%
\bibitem{alr} V. Abatangelo, B. Larato and  L.A. Rosati, Unitals in planes derived from Hughes planes, {\em  J. Combin. Inform. System Sci.} {\bf 15} (1990), 151--155.
%
\bibitem{alb} A.A. Albert, Finite division algebras and finite planes, {\em Proc. Sympos. Appl. Math.} {\bf 10} (1960),  53--70.
%
\bibitem{andre} J. Andr\'e, \"Uber nicht-Desarguessche Ebenen mit transitiver Translationsgruppe, {\em Math. Z.} {\bf 60} (1954),  156--186. 
%
\bibitem{baba} B. Bagchi and S. Bagchi, Designs from pairs of finite fields, I. A cyclic unital $U(6)$  and other regular Steiner 2-designs, {\em  J. Combin. Theory Ser. A} {\bf 52} (1989), 51--61.
%
\bibitem{barlun}   A. Barlotti and G. Lunardon,  Una classe di unitals nei $\Delta$-piani, {\em Riv. Mat. Univ. Parma (4)} {\bf 5} (1979), 781--785. 
%
\bibitem{bar} S.G. Barwick,  Unitals in the Hall plane, {\em J. Geom.} {\bf 58} (1997),  26--42.
%
\bibitem{bcq} S.G. Barwick, L.R.A. Casse and  C.T. Quinn, The Andr\'e/Bruck and Bose representation in ${\PG}(2h,q)$: unitals and Baer subplanes, {\em Bull. Belg. Math. Soc. Simon Stevin}  {\bf 7}  (2000),  173--197.
%
\bibitem{be} S. Barwick and G. Ebert,  Unitals in projective planes,  Springer, New York, 2008.
%
\bibitem{bjj} M. Biliotti, V. Jha and N.L. Johnson, Foundations of Translation Planes, Pure and Applied Mathematics, Vol. 243, Marcel Dekker, New York, Basel, 2001.
%
\bibitem{bro} A.E. Brouwer,  Some unitals on 28 points and their embeddings in projective planes of order 9, {\em  Geometries and Groups},  Springer Lecture Notes in Mathematics,  {\bf 893} (1981), 183--188. 
%
\bibitem{bro2} A.E. Brouwer, A unital in the Hughes plane of order nine, {\em Discrete Math.} {\bf 27}
(1989), 55--56.
%
\bibitem{bb} R.H.~Bruck and R.C.~Bose, The construction of translation planes from projective spaces, {\em J. Algebra} {\bf 1} (1964), 85--102.
%
\bibitem{bb2} R.H.~Bruck and R.C.~Bose, Linear representation of projective planes in projective spaces, {\em J. Algebra} {\bf 4} (1966), 117--172.
%
\bibitem{bt} A.A. Bruen and J.A. Thas, Blocking sets, {\em  Geometriae Dedicata} {\bf 6} (1977),  193--203.
%
\bibitem{b} F. Buekenhout, Existence of unitals in finite translation planes of order $q^{2}$ with a kernel of order $q$, {\em Geometriae Dedicata} {\bf 5} (1976),  189--194.
%
\bibitem{cp} J. Cannon and C. Playoust, An Introduction to MAGMA, University of Sydney Press, 193.
%
 \bibitem{dem} P. Dembowski, Finite geometries, Springer-Verlag, New York, 1968.
 %
 \bibitem{deha} M. J. de Resmini and N. Hamilton, Hyperovals and unitals in Figueroa
planes,  {\em  European J. Combin.}  {\bf 19} (1998), 215--220.
 %
\bibitem{dov} J. Dover, A family of non-Buekenhout unitals in the Hall planes, in \enquote{Mostly finite geometries (Iowa City, IA, 1996)}, {\em Lecture Notes in Pure and Appl. Math.} {\bf 190} (1997), 197--205. 
%
\bibitem{drudge} K. Drudge, On the orbits of Singer groups and their subgroups, {\em Electron. J. Combin.}  {\bf 9}  (2002), 10 pp. (electronic).
%
\bibitem{gan2} M.J.  Ganley, A class of unitary block designs, {\em Math. Z.} {\bf 128} (1972), 34--42.
%
\bibitem{gru2} K. Gr\"uning, Das kleinste Ree-Unital, {\em Arch. Math.} {\bf 46} (1986), 473--480.
%
\bibitem{gru} K. Gr\"uning,  A class of unitals of order $q$ which can be embedded in two different planes of order $q^2$, {\em  J. Geom.} {\bf 29} (1987),  61--77. 
%
\bibitem{hltw}  A.M.W. Hui, H.F. Law, Y.K. Tai and P.P.W. Wong, A note on unitary polarities in finite Dickson semifield planes, {\em J. Geom.} {\bf 106} (2015), 175--183.
%
 \bibitem{kan} W.M. Kantor, Some generalized quadrangles with parameters $q^2,q$, {\em Math. Z.} {\bf 192} (1986),  45--50.
%
\bibitem{ml} M. Lavrauw,  Characterizations and properties of good eggs in $\PG(4m-1,q)$, $q$ odd, {\em  Discrete Math.} {\bf 301} (2005),  106--116. 
%
 \bibitem{lp} M. Lavrauw and T. Penttila,  On eggs and translation generalised quadrangles, {\em  J. Combin. Theory Ser. A} {\bf 96} (2001),  303--315.
%
\bibitem{lun2} G. Lunardon, Flocks, ovoids of $Q(4,q)$ and designs, {\em  Geom. Dedicata} {\bf 66} (1997), 163--173.
%
\bibitem{lun} G. Lunardon, Normal spreads, {\em Geom. Dedicata} {\bf 75} (1999), 245--261.
%
\bibitem{math} R. Mathon, Constructions for cyclic Steiner 2-designs, {\em Ann. Disc. Math.} {\bf 34} (1987), 353--362.
%
\bibitem{mp} F. Mazzocca and O. Polverino, Blocking sets in ${\PG}(2,q^n)$ from cones of ${\PG}(2n,q)$, {\em J. Algebraic Combin.} {\bf 24} (2006),  61--81. 
%
\bibitem{mps} F. Mazzocca, O. Polverino and L. Storme, Blocking sets in $\PG(r,q^n)$, {\em Des. Codes Cryptogr.} {\bf 44} (2007), 97--113. 
%
\bibitem{m} R. Metz, On a class of unitals, {\em Geom. Dedicata} {\bf 8} (1979), 125--126.
%
\bibitem{pt} S.E Payne and J.A.  Thas, Finite generalized quadrangles, Second edition, in: EMS Series of Lectures in Mathematics, Z\"urich, 2009. 
%
\bibitem{pw}  T. Penttila and B. Williams, Ovoids in parabolic spaces, {\em Geom. Dedicata},  {\bf 82} (2000), 1--19. 
%
\bibitem{rin} G. Rinaldi,  Hyperbolic unitals in the Hall planes, {\em J. Geom.} {\bf 54}
 (1995), 148--154.
 %
\bibitem{ros} L.A. Rosati, Disegni unitari nei piani di Hughes, {\em Geom. Dedicata} {\bf 27} 
(1988), 295--299.
%
\bibitem{rsv} S. Rottey, J. Sheekey and G. Van de Voorde, Subgeometries in the Andr\'e/Bruck-Bose representation, \emph{Finite Fields Appl.} {\bf 35} (2015), 115--138.

\bibitem{seib} M. Seib,  Unit\"are Polarit\"aten endlicher projektiver Ebenen, {\em Arch. Math.} {\bf 21} (1970), 103--112. 
%
\bibitem{cgmssw} T. Sz\H onyi, A. Cossidente, A. G\'acs, C.  Mengyn, A. Siciliano and  Z. Weiner, On large minimal blocking sets in $\PG(2,q)$,  {\em J. Combin. Des.} {\bf 13} (2005), 25--41. 
%
\bibitem{thas} J.A. Thas, Geometric characterization of the $[n-1]$-ovaloids of the projective space $\PG(4n-1,q)$, {\em Simon Stevin} {\bf 47} (1974), 97--106.
%
\bibitem{thas2} J.A. Thas, Generalized quadrangles of order $(s,s^2)$, II, {\em J. Combin. Theory Ser. A} {\bf 79} (1997), 223--254.
%
%
\end{thebibliography}
\end{document}